\documentclass[letterpaper, 10 pt, conference]{ieeeconf}  

\IEEEoverridecommandlockouts                              
\usepackage{graphics} 
\usepackage{epsfig} 
\usepackage{mathptmx} 
\usepackage{times} 
\usepackage{amsmath} 
\usepackage{amssymb}  
\usepackage[dvipsnames]{xcolor}
\usepackage{mathtools}
\usepackage{cleveref}
\usepackage{newtxmath}
\usepackage[most]{tcolorbox}
\usetikzlibrary{arrows,shapes,trees}
\usetikzlibrary{backgrounds,calc}

\newtheorem{theorem}{Theorem}
\newtheorem{proposition}{Proposition}
\newtheorem{corollary}{Corollary}

\newtheorem{definition}{Definition}
\newtheorem{remark}{Remark}

\newcommand\norm[1]{\left\lVert#1\right\rVert}
\newcommand\abs[1]{\lvert#1\rvert}

\title{\LARGE \bf
Coupling Induced Stabilization of Network Dynamical Systems and Switching
}

\author{Moise R. Mouyebe and Anthony M. Bloch
\thanks{Department of Mathematics, University of Michigan, Ann Arbor, MI 48109, USA
{\tt\small \{mmouyebe,abloch\}@umich.edu} Suppored in part 
by NSF grant DMS-2103026, and AFOSR grants FA 9550-24-1-0215 and 9550-23-1-0400 (MURI)}
}

\begin{document}

\maketitle
\thispagestyle{empty}
\pagestyle{empty}

\begin{abstract}

This paper investigates the stability and stabilization of diffusively coupled network dynamical systems. We leverage Lyapunov methods to analyze the role of coupling in stabilizing or destabilizing network systems.  We derive critical coupling parameter values for stability and provide sufficient conditions for asymptotic stability under arbitrary switching scenarios, thus highlighting  the impact of both coupling strength and network topology on the stability analysis of such systems. Our theoretical results are supported by numerical simulations.
\end{abstract}


\section{INTRODUCTION}

Network dynamical systems (NDS) are an essential modeling framework in various fields such as communication networks \cite{baggio2020efficient}, biological systems \cite{selimkhanov2014accurate}, and power grids \cite{galli2011grid}. These systems consist of interconnected dynamical nodes, where the state of each node is influenced by that of its neighbors. Stability in NDS is crucial, as instabilities such as oscillations or chaos can negatively impact their performance. However, achieving stability is particularly challenging due to the nonlinear and complex interactions between nodes. To address this challenge, coupling strategies have emerged as an  effective solution \cite{padmanaban2015coupling}, \cite{hens2022stabilization}. Adequate coupling between nodes can promote synchronization and emergent behaviors \cite{dong2021briefing}, \cite{ocampo2024strong}, thereby enhancing the overall network stability, even when individual node dynamics are unstable. On the contrary, ill-designed coupling can have the opposite effect, leading to instability in otherwise stable networks. This is somewhat  reminiscent of the potentially destabilizing effect of dissipation on Hamiltonian systems with symmetry \cite{bloch1994dissipation}. 

The concept of nonlinear stability in dynamical systems was initially defined by Lyapunov.  Sontag’s work \cite{Sontag1996} on input-to-state stability helped lay the foundation for understanding how external inputs affect the stability of interconnected systems. In such systems, the couplings between subsystems can either be positive or negative, depending on the nature of the connections. Stabilizing NDS requires a clear understanding of both the local dynamical behavior of individual nodes as well as their mutual interactions, and careful analysis and design of the coupling between nodes is essential \cite{pecora1998master}, \cite{jalili2013enhancing}. Coupling-induced stabilization occurs when appropriate couplings drive the system to a stable equilibrium as seen in multi-agent consensus in coupled oscillators \cite{OlfatiSaber2007}, \cite{Pikovsky2001}.
Coupling mechanisms in NDS can be classified into local coupling and global coupling. Local coupling refers to interactions between neighboring nodes, whereas global coupling influences the entire network. The stability of NDS depends on the network’s structural characteristics, and identifying the optimal coupling structure and strength is a significant challenge. Recent research has focused on distributed optimization techniques \cite{Li2018} for stabilizing networked systems. Meanwhile \cite{ocampo2024strong} showed that an appropriately designed coupling term could induce synchronization and enhance overall system stability. In engineering applications, such as power grids, specific coupling strategies have been used to synchronize generator units, illustrating the importance of coupling induced stabilization \cite{dong2021briefing}.

In real-world networks, interconnections between components often change over time. This introduces the notion of switching in NDS, which refers to the dynamic alteration of the coupling structure due to external influences or intrinsic system demands. For instance, in communication networks, the network topology may change as connections are established or severed, meanwhile  in power grids, operational requirements can alter the strength of coupling between nodes. These changes can destabilize the network, especially if the connectivity changes abruptly or unpredictably \cite{yang2015network}, \cite{moreau2005stability}. Therefore, it is essential to develop theoretical results that ensure the stability of network systems under arbitrary switching scenarios, thereby enhancing the network's resilience to disturbances.

In this paper, we harness Lyapunov methods to analyze the impact of coupling on the stability and stabilization of diffusively coupled dynamical systems. We derive expressions for critical coupling parameter values that determine the stability of the network system. Furthermore, we derive sufficient conditions to guarantee asymptotic stability of the system under arbitrary switching, and numerical simulations are used to validate the theoretical results.

The paper is structured as follows: Section 2 provides the necessary background material, Section 3 presents the mathematical framework for coupling-induced stabilization, Section 4 discusses switched NDS and their stability under arbitrary switching, and Section 5 presents numerical results. Finally, the paper concludes with potential directions for future research.


\section{Preliminaries}
 \subsection{Dynamical systems on Graphs}
 
 In this work, we restrict ourselves to the class of simple  and connected graphs both directed and undirected. Simple here refers to the graphs that do not have more than one edge between any two vertices and no edge starts and ends at the same vertex (i.e no self-loops). Furthermore, we'll assume that all the agents have the same dynamics and interact in the same way with one another.
 
 \paragraph{The Framework}
 
 Let $\mathbf{A}$ be the adjacency matrix of a graph of $N$ interacting dynamical units (agents) each with isolated (node) dynamics $\dot{x}_i = f(x_i) \quad  \forall i\in \{1,\cdots,N\}$, where $x_i \in \mathbb{R}^d$ is the state vector of agent $i$. Furthermore, we assume that the pairwise interaction between agents $i\ne j$ is captured by the same coupling function $\phi: \mathbb{R}^d \times  \mathbb{R}^d \mapsto  \mathbb{R}^d$. The resulting network system dynamics can be represented by the following system of ODEs:
 
 \begin{equation}\label{eq:networkDynamics}
    \dot{x}_i = f(x_i) +\sum_{j}^{N}\mathbf{A}_{ij}\phi(x_i,x_j) , \quad 1\le i \le N
\end{equation}

\paragraph{The Terminology}

We refer to the matrix $\mathbf{A}$ as the network  \emph{topology} or \emph{architecture} or  \emph{wiring}.  Meanwhile the mapping $\phi$ is referred to as the \emph{coupling}. 

In this paper, we are particularly interested in \emph{linear couplings}. These are coupling of the form 

\begin{equation}\label{eq:linear coupling}
    \phi(u,v) = M(\alpha u +\beta v)
\end{equation}

where $\alpha$ and $\beta$ are real numbers and $M$ is the \emph{coupling channel matrix} which in general is chosen to be a diagonal matrix with boolean entries. 

Let $\mathbf{x} = [x_1,\cdots,x_N]^T \in \mathbb{R}^{Nd}$ be the joint state vector, namely the state vector of the network system. If we define $\textbf{f}(\mathbf{x}) = [f(x_1),\cdots,f(x_N)]^T$, then (\ref{eq:networkDynamics}) can be compactly written using Kronecker product as follows: 
\begin{align}\label{eq:compactNwkDynamics}
    \dot{\mathbf{x}} &= \textbf{f}(\mathbf{x}) + \left(L_{\alpha\beta}\otimes M\right )\mathbf{x}
\end{align}

where 

\begin{align}\label{eq:couplingLaplacian}
    L_{\alpha\beta} &= \alpha D^{in} + \beta \mathbf{A}
\end{align}

$D^{in} = diag(\mathbf{A}\mathbf{1})$ is the \emph{in-degree} matrix of $\mathbf{A}$. We call the matrix $L_{\alpha\beta}$ the \emph{parametrized coupling Laplacian} or simply the \emph{coupling Laplacian} with \emph{coupling parameters} $\alpha$ and $\beta$. 

 The coupling Laplacian encodes information about the network topology and the linear coupling in a one-to-one fashion. In other words, a given coupling Laplacian will instantiate the linear coupling (up to a coupling channel matrix $M$, which without loss of generality can be  the identity matrix) as well as the network topology and vice-versa.

\paragraph{Equilibria of the Dynamics}

\begin{proposition}\label{prop:equilibria}
Let $x_0$ be an equilibrium state for the isolated dynamics $f$. 
\begin{enumerate}
    \item If  $\alpha = -\beta$, then the vector $\mathbf{1} \otimes x_0$ is an equilibrium for the network  dynamics (\ref{eq:compactNwkDynamics}) .
    \item If $x_0 = 0$, then the origin is an equilibrium for the network dynamics (\ref{eq:compactNwkDynamics}) for any $\alpha,\beta \in \mathbb{R}$.
\end{enumerate}
\end{proposition}

\begin{proof}
    The first assertion follows from the fact that  if $\alpha = -\beta$, then the coupling Laplacian has the \emph{zero-row-sum} property, namely $L_{\alpha\beta}\mathbf{1} = 0.$ The second assertion holds by linearity of the operator $M$.
\end{proof}

\subsection{The Linearized Dynamics}
\paragraph{Linearization near a Reference Trajectory}

\begin{proposition}\label{prop:variationalEqn}
    Suppose we're given an \emph{arbitrary} reference trajectory $\bar{\mathbf{x}}(t)$ of (\ref{eq:networkDynamics}), the variational equation describing the dynamics of infinitesimal perturbations is given by the following linear time-varying equation
\begin{multline}\label{eq:variationalEqn}
    \dot{\mathbf{\delta x}} = \left( \sum_{i=1}^N{\mathbf{E}_{ii}\otimes \frac{\partial f}{\partial x}} \left(\bar{\mathbf{x}}\right) + \sum_{i,j=1}^{N}{ \left\{ \mathbf{E}_{ii}\otimes \frac{\partial \phi}{\partial u }\left(\bar{\mathbf{x}},\bar{\mathbf{x}}\right) +  \right. }    \right. \\ 
    \left.\left. \mathbf{E}_{ij}\otimes \frac{\partial \phi}{\partial v}\left(\bar{\mathbf{x}},\bar{\mathbf{x}}\right) \right\}\right)\mathbf{\delta x}
\end{multline}
\end{proposition}

\begin{proof}
    Let $\mathbf{x}$ be a $\mathbf{\delta x}$ perturbed trajectory away from $\bar{\mathbf{x}}$.  Denote by $\{\mathbf{E}_{ij}\}_{1\le i,j \le N}$ the standard basis of the space Mat$_N(\mathbb{R})$ of square matrices with coefficients in the field $\mathbb{R}$. A Taylor expansion yields the following
    \begin{align}\label{TaylorExpanLinDynRefTraj}
        \begin{split}
            \dot{\bar{x}}_i + \dot{\delta x_i} &= f(\bar{x}_i + \delta x_i) + \sum_{j=1}^{N}{\mathbf{A}_{ij}\phi(\bar{x}_i+\delta x_i,\bar{x}_j+\delta x_j)},~ 1\le i \le N \\
            &= f(\bar{x}_i)+\frac{\partial f}{\partial x} (\bar{x}_i)\delta x_i + f_{2+}(\delta x_i) +  \sum_{j=1}^{N}{ \mathbf{A}_{ij}}\left\{ \phi(\bar{x}_i,\bar{x}_j)+ \right.\\
           &\left. \frac{\partial \phi}{\partial u}(\bar{x}_i,\bar{x}_j)\delta x_i + \frac{\partial \phi}{\partial v}(\bar{x}_i,\bar{x}_j)\delta x_j+ \phi_{2+}(\delta x_i,\delta x_j) \right\}
        \end{split}
    \end{align}
    Noting that $\bar{\mathbf{x}}$ is a solution of (\ref{eq:networkDynamics}) and ignoring the higher order terms, we recover (\ref{eq:variationalEqn}) as the compact version of (\ref{TaylorExpanLinDynRefTraj}).
\end{proof}


\begin{definition}
    The synchronization manifold of the system (\ref{eq:networkDynamics}) is the submanifold of the joint state space along which the states of the isolated dynamics coincide. More specifically we have
    \begin{align}\label{def:SyncManifold}
        \mathcal{S} \coloneq \left\{ \mathbf{x}=(x_1,\cdots,x_N) \in \mathbb{R}^{Nd} ~\rvert~ x_i = x_j, ~\forall~i,j = 1,\cdots, N  \right\}
    \end{align}
\end{definition}

\begin{corollary}
    If the reference trajectory $\bar{\mathbf{x}}$ in Proposition \ref{prop:variationalEqn} is constrained to the synchronization manifold $\mathcal{S}$, then (\ref{eq:variationalEqn}) simplifies to 

    \begin{align}\label{eq:varEqn_on_SyncManifold}
    \dot{\mathbf{\delta x}} = \left (  I_{N}\otimes \frac{\partial f}{\partial x}  \left(\bar{x}\right)  +  D^{in}\otimes \frac{\partial \phi}{\partial u}\left(\bar{x},\bar{x}\right) +  \mathbf{A}\otimes \frac{\partial \phi}{\partial v}\left(\bar{x},\bar{x}\right)     \right )\mathbf{\delta x} 
\end{align}

In particular, if the coupling $\phi$ is also constrained to be linear as in (\ref{eq:linear coupling}), then  system (\ref{eq:varEqn_on_SyncManifold}) further simplifies to 

\begin{align}\label{eq:varEqn_on_SyncManifold_with_LinCoupling}
    \dot{\mathbf{\delta x}} = \left (  I_{N}\otimes \frac{\partial f}{\partial x}  \left(\bar{x}\right)  +  L_{\alpha\beta}\otimes M
        \right )\mathbf{\delta x} 
\end{align}
\end{corollary}

\begin{remark}
    Note that equations (\ref{eq:variationalEqn}), (\ref{eq:varEqn_on_SyncManifold}) and (\ref{eq:varEqn_on_SyncManifold_with_LinCoupling}) are fundamentally linear time-varying (LTV) systems. However, if the reference trajectory $\bar{\mathbf{x}}$ specializes to an equilibrium (see   Proposition \ref{prop:equilibria}), then these equations effectively become linear time-invariant (LTI) systems, which in general are more analytically tractable.
\end{remark}

\paragraph{The Spectrum of the Jacobian}

\begin{proposition}\label{prop:Jaobian_Spectrum}
    Let $\mathbf{x}_0 \in \mathcal{S}$ be an equilibrium  of the dynamics (\ref{eq:networkDynamics}), and assume that $M = I_d$. The spectrum $\sigma(\mathbf{J})$ of the Jacobian
    $\mathbf{J} = I_{N}\otimes \frac{\partial f}{\partial x}  \left(x_0\right)  +  L_{\alpha\beta}\otimes I_d$  of the system is given by:
    \begin{align}\label{eq:JacobianSpectrum_nwkDyn}
        \begin{split}
            \sigma(\mathbf{J}) = \sigma\left(\frac{\partial f}{\partial x}  \left(x_0\right)\right) + \sigma\left(L_{\alpha\beta}\right)
        \end{split}
    \end{align}
\end{proposition}

\begin{proof}
    The proof is fairly straightforward, and it relies on the  Schur decomposition of the matrices $\frac{\partial f}{\partial x}  \left(x_0\right)$ and $L_{\alpha\beta}$. See \cite{horn1994topics} for more general details.
\end{proof}

\section{Coupling Induced Stabilization of Network Dynamics}

\subsection{Stability Conditions}\label{sec:stabilityCond}

From now on, we assume that the equilibrium of the dynamics  is located at the origin. This can always be achieved using an appropriate change of coordinates. Furthermore, All stability notions discussed will be local unless otherwise indicated.

\begin{proposition}\label{prop:NecessaryCond_stability-NwkDyn}
    Assume the coupling Laplacian $L_{\alpha\beta}$ has an eigenvalue in the Right Half Plane (RHP).  The stability of the node dynamics is necessary for that of the network dynamics (\ref{eq:compactNwkDynamics}).
\end{proposition}

\begin{proof}
    Note that under the assumption, there exists $\lambda_0 \in \sigma\left(L_{\alpha\beta}\right)$ such that $\operatorname{Re}({\lambda_0}) > 0$. Now, if the node dynamics is not stable, then there exists some $\nu_0 \in \sigma\left(\frac{\partial f}{\partial x}  (0)\right)$ with $\operatorname{Re}({\nu_0}) \ge 0.$ Conclude with (\ref{eq:JacobianSpectrum_nwkDyn}) that the spectrum of the Jacobian $\mathbf{J}$ must intersect the RHP. Hence (\ref{eq:compactNwkDynamics}) cannot be stable.
\end{proof}

We now assume that $\abs{\alpha} = \abs{\beta}$ for the remainder of this section. Under this assumption, the dynamics (\ref{eq:compactNwkDynamics}) only depends on one parameter whose magnitude is called the \emph{coupling strength}. This parameter has a strong effect on the network dynamics. For instance, it can potentially \emph{destabilize} a network of stable isolated dynamics,  or on the contrary \emph{stabilize} a network of unstable node dynamics as we'll see momentarily. This shows that the condition in Proposition \ref{prop:NecessaryCond_stability-NwkDyn} is only necessary but not sufficient in general. Let's point out that the tuning of this \emph{control parameter} can generate interesting bifurcations in the network dynamics, but this is beyond the scope of the present communication.

\begin{remark}\label{rmk:diffusive-infusive-coupling}
    The case $\beta = -\alpha$ is often referred to as \emph{ diffusive coupling} in the literature. So, we'll refer to  the case $\beta = \alpha$ as \emph{signless-diffusive coupling}.
\end{remark}

Let $L^{\pm} = D^{in}\pm A$ be the \emph{network coupling Laplacian}, where we lumped together in a single notation the \emph{"classical" Laplacian} $L^{-} = D^{in} -  A$ and the  \emph{signless Laplacian} $L^{+} = D^{in} +  A$. The latter unlike the former, does not frequently appear in the literature, and  one of the earliest research work pertaining to it can be traced back to \cite{desai1994characterization}.

  Let $\sigma\left(L^{\pm}\right):= \left\{\lambda^{\pm}_1,\cdots,\lambda^{\pm}_N\right\}$ be the associated spectra. The spectrum of the network Jacobian (\ref{eq:JacobianSpectrum_nwkDyn}) now takes the form: 
\begin{align}\label{eq:Jacobian-Ntwk_plus_minus}
    \sigma(\mathbf{J}) = \left\{\nu_i +\alpha \lambda^{\pm}_j ~\rvert~ 1\le i \le d, \quad 1\le j \le N \right\}
\end{align}

Define $\nu_{\max}:=\max\left\{ \operatorname{Re}({\nu_i})  ~\rvert~ 1\le i \le d \right\}$, and similarly for $\lambda^{\pm}_{\min}$ and $\lambda^{\pm}_{\max}$. We assume that $\nu_{\max}$ does not vanish unless otherwise indicated. We have the following proposition.\vspace{5pt}

\begin{proposition}\label{prop:Stability-Interval_NtwkDyn}
      Suppose the node dynamics is  stable. There exists a critical coupling parameter value $\alpha_{c} > 0$, such that the network (\ref{eq:compactNwkDynamics}) is  stable (exponentially) for $ \alpha \in (-\infty, \alpha_{c})$ and unstable for $ \alpha \in (\alpha_{c},\infty).$ 
\end{proposition}

\begin{proof}
        The spectrum of $L^{\pm}$ lies in the open right half plane union the origin, thus $\operatorname{Re}(\lambda_j^{\pm}) \ge 0$ for all $j$.  If $\alpha \le 0$, then $\operatorname{Re}(\nu_i+\alpha\lambda_j^{\pm}) < 0$ for all $1\le i\le d$ and $1\le j\le N$. Hence the network is  stable for all $\alpha \in (-\infty, 0]$. If however $\alpha > 0$, then   $\nu_{\max}+\alpha\lambda_{\max}^{\pm}$ is the  top  element of the bounded lattice $\left(\operatorname{Re}(\sigma(\mathbf{J})), \le  \right)$. Therefore for $\alpha_c:= -\frac{\nu_{\max}}{\lambda_{\max}^{\pm}}$, it follows that the network is  stable for all $\alpha \in (0, \alpha_c)$ and unstable for $\alpha \in (\alpha_c, \infty)$. Note that $\lambda_{\max}$ does not vanish since  the network topology is connected.
\end{proof}

\subsection{Coupling Induced Stabilization}
In light of Proposition \ref{prop:NecessaryCond_stability-NwkDyn} which gives a necessary condition for the stability of the network, the natural question that arises is that of whether a network of unstable node dynamics can be stabilized by means of mere coupling? We answer this question in the \emph{affirmative} provided that the \emph{spectrum of the network Laplacian does not contain the origin}. Otherwise, active control such as feedback is required to stabilize the network. 

\begin{proposition}\label{prop:Critical_Stabilizing_coupling_val}
      Assume $0 \notin \sigma(L^{\pm})$ and the node dynamics is unstable. There exits a critical coupling parameter value $\alpha_c < 0$ below which the network is exponentially stable and above which the network is unstable.
\end{proposition}

\begin{proof}
    The argument of the proof here is similar to that of Proposition \ref{prop:Stability-Interval_NtwkDyn}, but noting that $\nu_{\max}+\alpha\lambda_{\min}^{\pm}$ is now the maximal element of the ordered set $\left(\operatorname{Re}(\sigma(\mathbf{J})), \le  \right)$. Here we have $\alpha_c:= -\frac{\nu_{\max}}{\lambda_{\min}^{\pm}}$.  Furthermore,  note that $\lambda_{\min}^{\pm}$ does not vanish since the origin is excluded from the geometry of the spectrum of $L^{\pm}$.
\end{proof}

\begin{proposition}\label{prop:Impossible-coupling-Stabilization}
   Suppose the node dynamics unstable. If the spectrum of the network Laplacian contains the origin, then the network cannot be stabilized by coupling alone.
\end{proposition}
\begin{proof}
    It follows from  (\ref{eq:Jacobian-Ntwk_plus_minus}) that the unstable manifold is nontrivial,  irrespective of the coupling parameter $\alpha$. 
\end{proof}

Since  $0 \in \sigma\left(L^-\right)$, we always need active control (e.g. feedback,$\cdots$ etc) to stabilize a network of unstable  diffusively coupled  dynamics.

On the other hand, if $\nu_{\max} $ vanishes then the dynamics on the \emph{center manifold} of the network system becomes relevant especially for $\alpha < 0$. Hence a direct corollary of Proposition \ref{prop:Impossible-coupling-Stabilization} can  be stated in a more evocative way as follows: \emph{Diffusion impedes coupling-induced stabilization of a network system}.

\subsection{Connection to Bipartite Graphs}

The  network Laplacian $L^-$ has the zero-row sum property, thus active control is needed to stabilize a network of unstable systems in this configuration (see Proposition \ref{prop:Impossible-coupling-Stabilization}). How about $L^+$?  

According to Proposition \ref{prop:Critical_Stabilizing_coupling_val}, as long as the spectrum of the network Laplacian  $L^+$ does not contain the origin, then coupling-induced stabilization is feasible. So, in this section we're interested in characterizing what kind of network topology yields a network Laplacian $L^+$ whose spectrum avoids the origin. The answer is given by the following proposition.

\begin{proposition}\label{prop:L-plus-Laplacian-inveritibility}
    The network Laplacian $L^+$ is  invertible if and only if the underlying graph is non-bipartite.
\end{proposition}

\begin{proof}
    This is a consequence of Proposition 2.1 in \cite{desai1994characterization}.
\end{proof}

\begin{proposition}
    Let $P_n$, $S_n$ and $C_n$  denote the Path, Star and Cycle graph on $n$ nodes respectively \cite{newman2018networks}. The following  statements about signless-diffusively coupled systems can be inferred from  \cref{prop:Critical_Stabilizing_coupling_val,prop:Impossible-coupling-Stabilization,prop:L-plus-Laplacian-inveritibility}.

    \begin{enumerate}
    \item One can stabilize a network of unstable node dynamics on $C_{2n+1}$ using coupling alone.
        \item  One cannot stabilize a network of unstable node dynamics on $P_n$, $S_n$ and $C_{2n}$ using coupling alone.
    \end{enumerate}
\end{proposition}

\section{Switched Network Dynamical Systems} In this section, we are  interested in the dynamical system (\ref{eq:compactNwkDynamics}) when the underlying graph topology $G$ is allowed to change over time. 
We consider a collection of linearly coupled network dynamical systems given by (\ref{eq:compactNwkDynamics}) where the underlying connectivity topology takes values in a \emph{finite} index set $\mathcal{G}:= \{G_1,\cdots, G_m\}$ of connected graphs defined on the same vertex set $V$. 

\subsection{ The Formulation}

A switched network system is made of two components:

\begin{enumerate}
    \item  A collection of vector fields $F_{G_t}: \mathbb{R}^{Nd} \rightarrow \mathbb{R}^{Nd}$. In this work, these are copies of  (\ref{eq:compactNwkDynamics}), namely $F_{G_t}(\textbf{x}) = \textbf{f}(\textbf{x})+\left(L_{G_t}\otimes M\right)\textbf{x}$, where $L_{G_t}$ is the coupling Laplacian associated with   the graph $G_t \in \mathcal{G}$. These are called \emph{modes} or \emph{subsystems} of the switched network system.
    \item A function of time $\sigma$ called \emph{switching signal} which specifies which mode is active at any given time instant, namely $\sigma: [0,\infty) \rightarrow \mathcal{G}$.  The main requirement on this switching signal is that on any bounded time interval, the number of switches be finite in order to avoid "unpleasant" phenomena such as \emph{chattering}. In general, $\sigma$ will be assume to be piecewise constant.
\end{enumerate}

The switched network dynamics takes the following form:

\begin{align}\label{eq:Switched-Nwk-syst}
    \dot{\textbf{x}}(t) = F_{\sigma(t)}\left(\textbf{x}(t)\right)
\end{align}

where $F_{\sigma(t)}  = \textbf{f}(\mathbf{x})+ \left(L_{\alpha\beta}^{\sigma(t)}\otimes M\right)\textbf{x}$~, with $\sigma(t) = G_t \in \mathcal{G}$ for all $t\in [0,\infty)$.

System (\ref{eq:Switched-Nwk-syst}) falls within the broader class of switched systems for which there are some available general conditions that guarantee stability under arbitrary switching. See for instance \cite{liberzon2003switching} for a good presentation of this topic.

Let's assume that the isolated dynamic has an equilibrium at the origin,  then it follows from Proposition \ref{prop:equilibria} that each mode $F_{\sigma}$ also has an equilibrium at the origin. We're interested in the stability (asymptotic)  properties of this equilibrium for the switched system (\ref{eq:Switched-Nwk-syst}). 


One of the most studied stability properties for switched systems like (\ref{eq:Switched-Nwk-syst}) is \emph{Global Uniform Asymptotic Stability} (GUAS) which stipulates that (\ref{eq:Switched-Nwk-syst}) is GUAS if there exists a class $\mathcal{KL}$ \cite{Khalil:1173048} function $\beta$ such that for all solutions of (\ref{eq:compactNwkDynamics}) the following holds:

\begin{align}\label{in:GUAS}
    |\textbf{x}(t)| \le \beta(|\textbf{x}(0)|,t)  \quad \forall t\ge 0, \quad \forall \sigma(\cdot)
\end{align}


In particular, if $\beta(x,t) = ce^{-s t}x$ for some constants $c,s>0$, then we say that (\ref{in:GUAS}) is Globally Uniformly Exponentially Stable (GUES). It's important to note that one can also  relax the  \emph{Global} requirement to only consider the \emph{Local} requirement in which the bound (\ref{in:GUAS}) above is only valid in a neighborhood of the origin.

\subsection{Stability Analysis of the Switched Network System}

In this section, we examine conditions under which the switched system (\ref{eq:Switched-Nwk-syst}) possesses a common Lyapunov function.

\paragraph{Common Quadratic Lyapunov Function}

The following proposition gives a sufficient condition  for the existence of a common quadratic Lyapunov function when the index set $\mathcal{G}$ contains undirected graphs only.

\begin{theorem}\label{control::thm:Quadratic-Lyapunov-Function__undirect_graphs}
    Suppose the node dynamics has a quadratic Lyapunov function. If  the coupling Laplacian $L_{\alpha\beta}^{\sigma(t)}$ satisfies $\alpha < 0 $ and $|\alpha| \ge |\beta|$ for all $t\ge0$,  then  the switched network system possesses a common quadratic Lyapunov function.
\end{theorem}

\begin{proof}
Note that for a fixed $t \ge$,  if $\alpha <0$ and $|\alpha| \ge |\beta|$ then the spectrum of the coupling Laplacian $L_{\alpha\beta}$ is contained in the left half-plane $\mathbb{C}_-$. In particular, if the underlying graph is undirected then $L_{\alpha\beta}$ is negative semi-definite. Now, let $v(x) = x^TPx$ be a quadratic Lyapunov function for the node dynamics, where $x\in \mathbb{R}^d$ and $P$ is a positive definite. The quadratic form $\mathcal{V}(\mathbf{x}) = \mathbf{x}^T(I_N\otimes P)\mathbf{x}$  is a common Lyapunov function  for the switched network system. In fact, $\mathcal{V}$ is evidently a positive definite. Its Lie derivative along  (\ref{eq:Switched-Nwk-syst}) is given by $\Dot{\mathcal{V}}(\mathbf{x}) = 2\mathbf{x}^TI_{N}\otimes P(\mathbf{f}(\mathbf{x})+L_{\alpha\beta}^{\sigma}\otimes M)\mathbf{x} = 2\sum_{i=1}^N{\Dot{v}(x_i)} + 2\mathbf{x}^T(L_{\alpha\beta}^{\sigma}\otimes PM)\mathbf{x}< 2\mathbf{x}^T(L_{\alpha\beta}^{\sigma}\otimes PM)\mathbf{x}$. 
In particular, if $M$ is chosen such that the product $PM$ is positive semi-definite (e.g. $M = I_d$), then we have the following uniform estimate: $\Dot{\mathcal{V}}(\mathbf{x}) < \underset{\sigma}{\sup}\{2\lambda_{\max}(L_{\alpha\beta}^{\sigma}\otimes PM) \}\mathbf{x}^T\mathbf{x} ~$. The result follows from compactness of the index set $\mathcal{G}$.
    
\end{proof}

\begin{corollary}
    If the node dynamics possesses a quadratic Lyapunov function and $\alpha <0$, then the  diffusive and signless-diffusive (see Remark \ref{rmk:diffusive-infusive-coupling}) coupled switched network system are GUES. 
\end{corollary}\vspace{5pt} 

\paragraph{Asymptotic Stability}

In the previous section, we gave a sufficient condition for the switched network system to be \emph{exponentially stable} under the assumption that the node dynamics has a \emph{quadratic Lyapunov function}.  In this section, we aim to \emph{relax} the quadratic Lyapunov function requirement. Consequently, we  expect the conclusion to be weaker than exponential stability. In fact, we show that if the node dynamics has a certain \emph{quadratic-type Lyapunov function} then the resulting switched network system is \emph{asymptotically stable}. Let's start with some definitions

\begin{definition}
    A function $\gamma: [0,a) \mapsto \mathbb{R}_+$ is said to be of class $\mathcal{K}$ if it's continuous, $\alpha(0)=0$ and  strictly increasing. It's said to be of class $\mathcal{K}_{\infty}$ if $a = \infty$  and in addition, it's of class $\mathcal{K}$ and radially unbounded, namely $\gamma(r) \rightarrow \infty$ as $r \rightarrow \infty$.
\end{definition}

\begin{definition}
    Let $x=0$ be an equilibrium point for the nonlinear system $\dot{x} = f(x)$. A Lyapunov function $V$ is said to be of quadratic-type for the system if there exist class $\mathcal{K}$ functions $\lambda_1,\lambda_2$ and $\gamma$ defined on $[0,a)$ and positive constants $c_1$ and $c_2$ such that $\forall~ x\in D = B_{\epsilon}(0) \subset \mathbb{R}^d$ with $\epsilon < a$, we have
    \begin{align}\label{control::eq:quadratic-type-LF}
        \begin{split}
            \lambda_1(\norm{x}) \le V(x) \le \lambda_2(\norm{x})\\
            \frac{\partial V}{\partial x} f(x) \le -c_1\gamma^2(\norm{x}) \\
            \norm{\frac{\partial V}{\partial x}} \le c_2\gamma(\norm{x})
        \end{split}
    \end{align}
\end{definition}

The following proposition gives a sufficient condition  for the existence of a common Lyapunov function ensuring \emph{asymptotic stability} for the switched network system.

\begin{theorem}\label{control::thm:Quad-type-common-Lyapunov}
    Consider the switched system (\ref{eq:Switched-Nwk-syst}), and suppose the node dynamics has a quadratic-type Lyapunov function  such that $\gamma(r) \le \delta r$ on the interval $[0,a)$ for some constants $a,\delta >0$. Let $k_{in}^{\sigma,i}$ and $k_{out}^{\sigma,i}$ denote the in-degree and out-degree respectively of node $i$  in the underlying graph $\sigma(t) \in \mathcal{G}$. If  the coupling Laplacian $L_{\alpha\beta}^{\sigma}$ satisfies $\alpha <0$ and $2\abs{\alpha}k_{in}^{\sigma,i} \ge \abs{\beta}(k_{in}^{\sigma,i}+k_{out}^{\sigma,i})~ \forall t\ge 0$ and $\forall~ i\in \{1,\cdots,N\}$, then  the switched network system possesses a common  Lyapunov function that guarantees asymptotic stability. Furthermore, if the function $\lambda_1$ in (\ref{control::eq:quadratic-type-LF}) is of class $\mathcal{K}_{\infty}$ and $a=\infty$ then the result holds globally.
\end{theorem}

\begin{proof}
    define $\mathcal{V}(\mathbf{x}):= \sum_{i=1}^{N}{V(x_i)}$ for $\mathbf{x} \in  D^N$, where $D$ is the domain in  $\mathbb{R}^d$ containing the origin and supporting $V$. The estimate $\mathcal{V}(\mathbf{x}) \ge \lambda_1\big(\max\{\norm{x_i},~1\le i \le N\}\big) = \lambda_1(\norm{\mathbf{x}})$ ensures positive definiteness of $\mathcal{V}$ since the function $\lambda_1$ is of class $\mathcal{K}$. The Lie derivative of $\mathcal{V}$ along  (\ref{eq:Switched-Nwk-syst}) is given by the expression:  $\dot{\mathcal{V}}(\mathbf{x}) = \sum_{i=1}^{N}\dot{V}(x_i)+\sum_{i,j=1}^{N}{\frac{\partial V}{\partial x_i}L_{\alpha\beta}^{\sigma,ij}Mx_j}$, where $L_{\alpha\beta}^{\sigma,ij}$ is the $ij$-entry of the coupling Laplacian $L_{\alpha\beta}^{\sigma}$ associated with the $\sigma$-subsystem. Taking into account (\ref{control::eq:quadratic-type-LF}) and the fact that the matrix $M$ is diagonal with boolean entries, we have the following estimate on the Lie derivative:  $\dot{\mathcal{V}}(\mathbf{x}) \le -\sum_{i=1}^{N}{c_1\gamma^2(\norm{x_i})} + \sum_{i,j=1}^{N}{c_2\abs{L_{\alpha\beta}^{\sigma,ij}}\gamma(\norm{x_i}) \norm{x_j} } $. Using the sublinearity of $\gamma$ on the interval $[0,a)$, we obtain the following quadratic upper-bound on the Lie derivative: $\dot{\mathcal{V}}(\mathbf{x}) \le \sum_{i=1}^{N}{\big( \sum_{j=1}^{N}{c_2\abs{L_{\alpha\beta}^{\sigma,ij}} \norm{x_i}\norm{x_j}} 
 -c_1\norm{x_i}^2\big)}$. Define the matrix $B_{\sigma}:= c_1I_N - c_2 L_{\alpha\beta}^{\sigma}$ where $I_N$ is the identity matrix of size $N$. We get the following estimate: $\dot{\mathcal{V}}(\mathbf{x}) \le -\frac{1}{2} \tilde{\mathbf{x}}^T\big( B_{\sigma}+B_{\sigma}^T \big)\tilde{\mathbf{x}}$, where $\tilde{\mathbf{x}} = [\norm{x_1},\cdots,\norm{x_N}]^T$. Since the index set $\mathcal{G}$ is compact, it suffices to establish positive definiteness of the matrix $B_{\sigma}+B_{\sigma}^T$. This in turn follows from establishing negative semidefiniteness of the matrix $L_{\alpha\beta}^{\sigma}+\big( L_{\alpha\beta}^{\sigma} \big)^T$ which owing to the Gershgorin circle theorem, is guaranteed whenever $\alpha <0$ and $2\abs{\alpha}k_{in}^{\sigma,i} \ge \abs{\beta}(k_{in}^{\sigma,i}+k_{out}^{\sigma,i})$. Furthermore, if the function $\lambda_1$ is of class $\mathcal{K}_{\infty}$, it follows that $\mathcal{V}$ must be radially unbounded, and the result holds globally.
\end{proof}

A few comments about Theorem \ref{control::thm:Quad-type-common-Lyapunov}  are in order. First, the scope of this theorem is wider than that of Theorem \ref{control::thm:Quadratic-Lyapunov-Function__undirect_graphs} since the latter requires the index set $\mathcal{G}$ to contain only \emph{undirected graph}. Next, the \emph{sufficient condition} in Theorem \ref{control::thm:Quad-type-common-Lyapunov}, namely $\alpha <0$ and $2\abs{\alpha}k_{in}^{\sigma,i} \ge \abs{\beta}(k_{in}^{\sigma,i}+k_{out}^{\sigma,i})$ is \emph{stronger} than that of Theorem \ref{control::thm:Quadratic-Lyapunov-Function__undirect_graphs} since the former recovers the latter namely $\alpha < 0 $ and $|\alpha| \ge |\beta|$ when the graphs in $\mathcal{G}$ are all undirected. Finally, if we only care about \emph{(signless) diffusive coupling}, namely if $\abs{\alpha} = \abs{\beta}$, then to ensure \emph{asymptotic stability} of the switched network system, it suffices that  $k_{in}^{\sigma,i} \ge k_{out}^{\sigma,i}$, for all graph $\sigma(t) \in \mathcal{G}$, and for all node $i \in \{1,\cdots,N\}$. This motivates the following proposition.

\begin{proposition}\label{prop:NonPosDivergenceGraphs}
    Assume the node dynamics possesses a quadratic-type Lyapunov function such that $\gamma(r) \le \delta r$ on the interval $[0,a)$ for some positive constants $a,\delta >0$. If $\alpha < 0$ and $k_{in}^{\sigma,i} \ge k_{out}^{\sigma,i}$, $\forall$ graph $\sigma(t) \in \mathcal{G}$,~ and $\forall$ node $i \in \{1,\cdots,N\}$, then the diffusive and  signless-diffusive   coupled switched network system are uniformly asymptotically stable. Furthermore, if the function $\lambda_1$ in (\ref{control::eq:quadratic-type-LF}) is of class $\mathcal{K}_{\infty}$ and $a=\infty$, then the result holds globally. 
\end{proposition}
\begin{proof}
    It follows from the proof of Theorem \ref{control::thm:Quad-type-common-Lyapunov}.
\end{proof}

 \begin{figure}[]
    \centering
\includegraphics[width=0.80\linewidth]{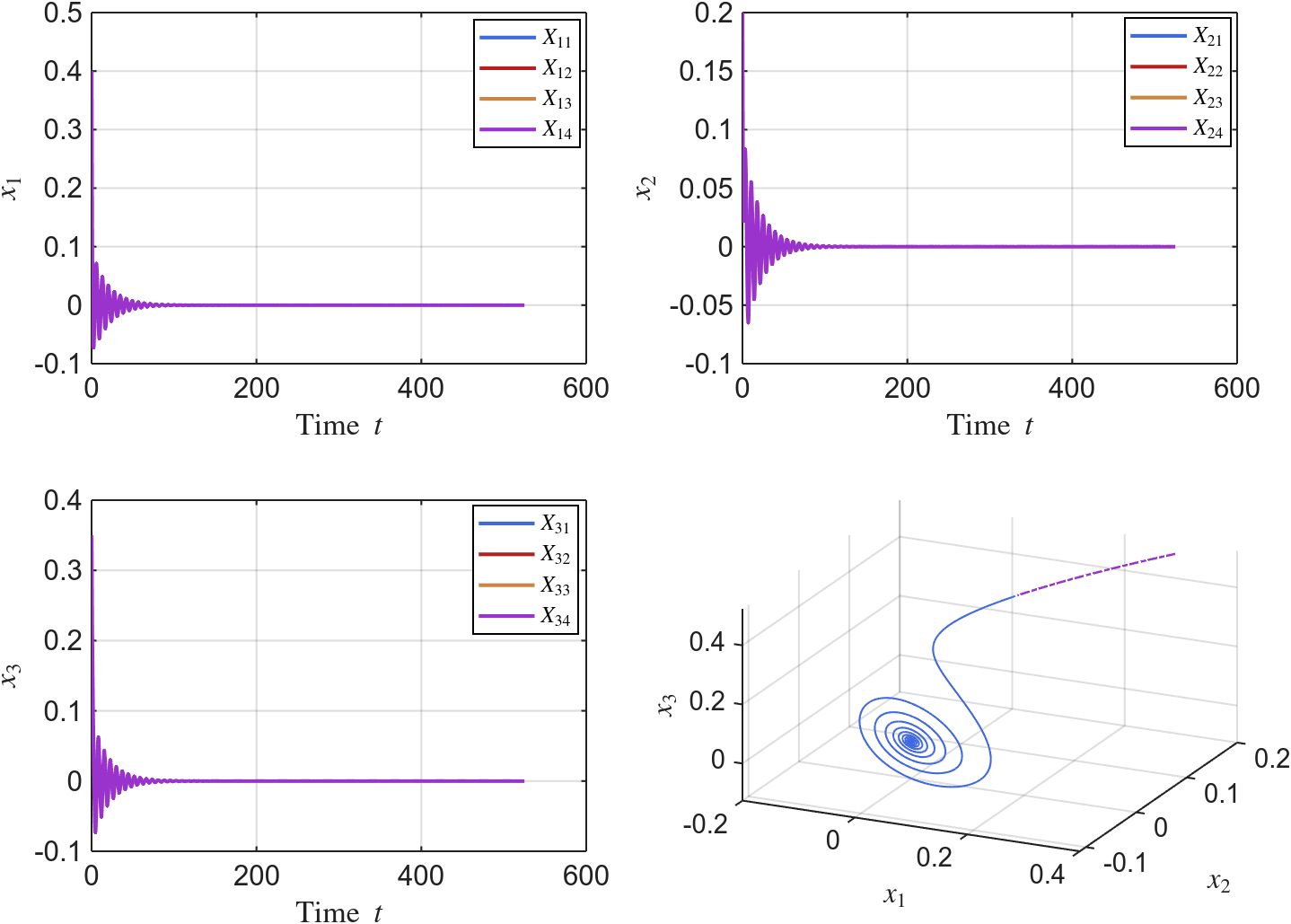} 
    \caption{Four diffusively coupled \emph{stable} Sprott systems \cite{sprott2010elegant} with parameter $\mu = 0.55$ on the bidirectional graph shown in Fig.  \ref{fig:BidirectionGrap-on-4-nodes}.  The coupling strength  is set to $\alpha = \alpha_c - \frac{1}{1000} = 0.0115$, where $\alpha_c = 0.0125$. The resulting network system is \emph{stable} as predicted by Proposition \ref{prop:Stability-Interval_NtwkDyn}. The top left, top right and bottom left plots show trajectories of the 1st, 2nd and 3rd channel respectively of each of the $4$  nodes. The bottom right plot shows the trajectory of the  node dynamics in 3D space. The purple dash-dotted  line shows the beginning segment of the trajectory to give a sense of orientation. The join state space is $12$-dimensional.}
    \label{fig:Stable_Sprott__sub_critical}
 \end{figure} 


\begin{figure}
    \begin{center}
        \includegraphics[width=0.80\linewidth]{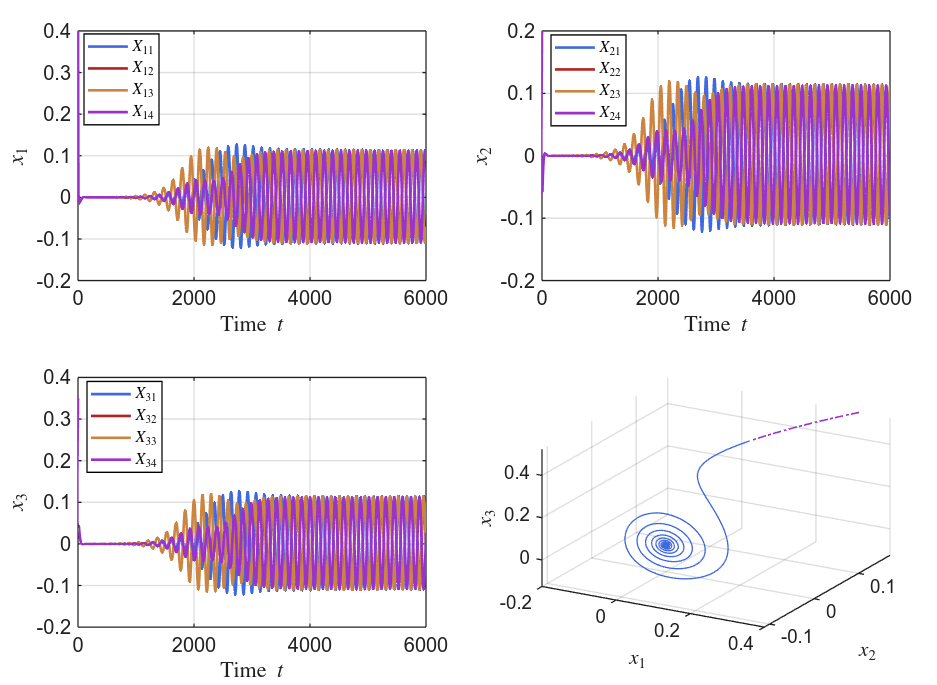} 
    \caption{Same setup as in Fig. \ref{fig:Stable_Sprott__sub_critical}., but here the  coupling strength is set to $\alpha = \alpha_c + \frac{1}{1000} = 0.0135$. The resulting network system is \emph{unstable} as predicted by Proposition \ref{prop:Stability-Interval_NtwkDyn}.}
    \label{fig:Stable_Sprott__super_critical}
    \end{center}
    
 \end{figure}

\tikzstyle{vertex} =[color=blue!50,fill,shape=circle]
\tikzstyle{edge} = [line width=4pt]

\begin{figure}
  \centering
    \resizebox{0.35\linewidth}{!}{
    \begin{tikzpicture}
    \node[vertex, text=white,scale=1] (v1) {1};
    \node[vertex,right of=v1, text=white,scale=1,node distance=100pt] (v2) {2};
    \node[vertex,below of=v2,
     text=white,scale=1, node distance=50pt] (v3) {3};
    \node[vertex,left of=v3, text=white,scale=1,node distance=100pt] (v4) {4};
    
    \begin{scope}[transparency group,opacity=.75]
    \tikzstyle{edge} = [line width=1.5pt]
    
    \draw[edge,style=<->](v1) -- (v2);
    \draw[edge,style=<->](v2) -- (v3);
    \draw[edge,style=<->](v3) -- (v4);
    \draw[edge,style=<->](v4) -- (v1);
    \draw[edge,style=<->](v1) -- (v3);
    \end{scope}
    \end{tikzpicture}
    }
    \caption{Bidirectional graph on a set of four nodes.}
    \label{fig:BidirectionGrap-on-4-nodes}
 \end{figure}
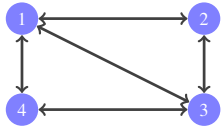

\section{Numerical Simulation}
In this section, we illustrate our theoretical results through numerical simulation.

First, we consider a static network of $4$ Sprott   systems \cite{sprott2010elegant} under stable ($\mu = 0.55$) and unstable ($\mu = 0$) regimes, and coupled over the bidirectional network shown in Fig. \ref{fig:BidirectionGrap-on-4-nodes}. The Sprott system  is an interesting dynamical system whose dynamical and controllability properties have been studied in \cite{mouyebe2024local}, \cite{kvalheim2021families}. It's a tractable version of the represilator \cite{elowitz2000synthetic}, a model of   synthetic genetic regulatory network ubiquitous in biology. In the stable regime,  we compute the critical coupling parameter to be $\alpha_c = 0.0125$, and both Fig. \ref{fig:Stable_Sprott__sub_critical}. and Fig. \ref{fig:Stable_Sprott__super_critical}. confirm that an $\epsilon = \frac{1}{1000}$ deviation away from this critical value can destabilize the network system. This ascertains the validity of  Proposition \ref{prop:Stability-Interval_NtwkDyn}. Likewise in the unstable regime which in fact is chaotic \cite{sprott2010elegant}, we compute $\alpha_c = -0.6545$, and  both Fig. \ref{fig:Unstable_Sprott__sub_critical}. and Fig. \ref{fig:Unstable_Sprott__super_critical}. show that one can indeed stabilize a network of signless-diffusively coupled unstable node dynamics accordingly.

\begin{figure}[]
    \centering
    \includegraphics[width=0.80\linewidth]{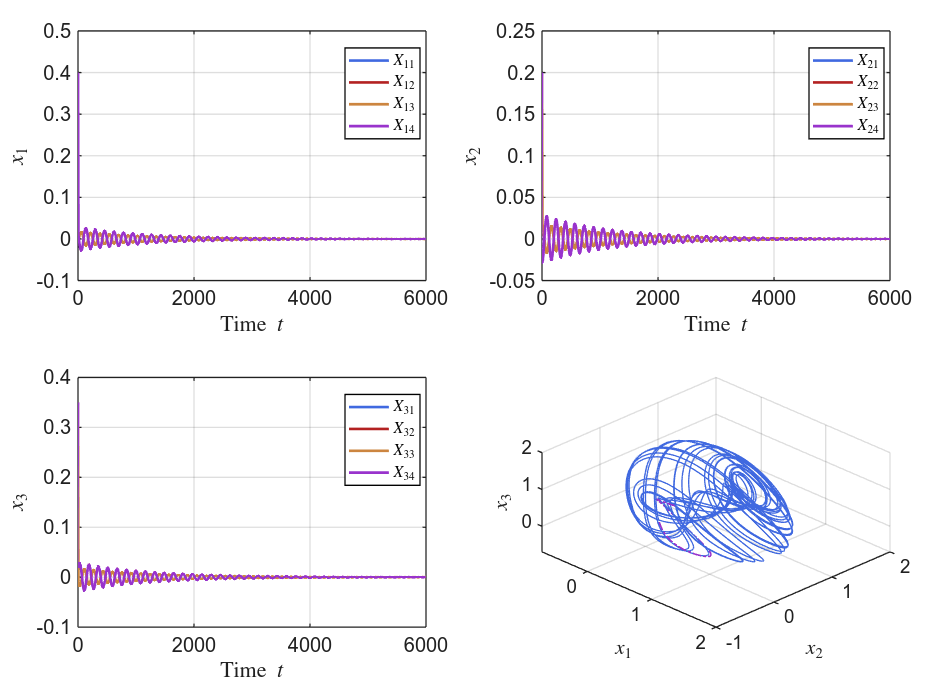} 
    \caption{Four \emph{signless-diffusively}  coupled \emph{unstable} Sprott systems \cite{sprott2010elegant} with parameter $\mu = 0$ on the bidirectional graph shown in Fig.  \ref{fig:BidirectionGrap-on-4-nodes}.  The coupling parameter is set to $\alpha = \alpha_c - \frac{1}{1000} = -0.6555$, where $\alpha_c = -0.6545$. The resulting network system is \emph{stable} as predicted by Proposition \ref{prop:Critical_Stabilizing_coupling_val}. The top left, top right and bottom left plots show trajectories of the 1st, 2nd and 3rd channel respectively of each of the $4$ nodes. The bottom right plot shows the trajectory of the node dynamics in 3D space. The purple dash-dotted line shows the beginning segment of the trajectory to give a sense of orientation. The join state space is $12$-dimensional. Note that the node dynamics in this regime (i.e. frictionless  Sprott system with $\mu = 0$) is in fact \emph{chaotic} \cite{sprott2010elegant}, yet with the \emph{right coupling} we're able to tame it down and effectively \emph{stabilize} the whole network system.}
    \label{fig:Unstable_Sprott__sub_critical}
 \end{figure}

Next, we signless-diffusively couple $4$ copies of the globally asymptotically stable node dynamics $\dot{x} = -x^3$  over what we call \emph{non-positive divergence}  graphs. These are directed graphs for which the out-degree of each node is no greater than its in-degree. The resulting network system is allowed to switch its connectivity topology at random a total of $7$ times during the course of the experiment within the set shown in Fig. \ref{fig:non-positive-Divergence-Graphs}. For coupling parameters $\alpha = -1 = \beta$, the hypothesis of Proposition \ref{prop:NonPosDivergenceGraphs} are satisfied, and Fig. \ref{fig:cubedScalar_a=-1_b=-1__stable}. confirms the theoretical predictions that the switched network system is indeed stable.

\begin{figure}[]
    \centering
    \includegraphics[width=0.80\linewidth]{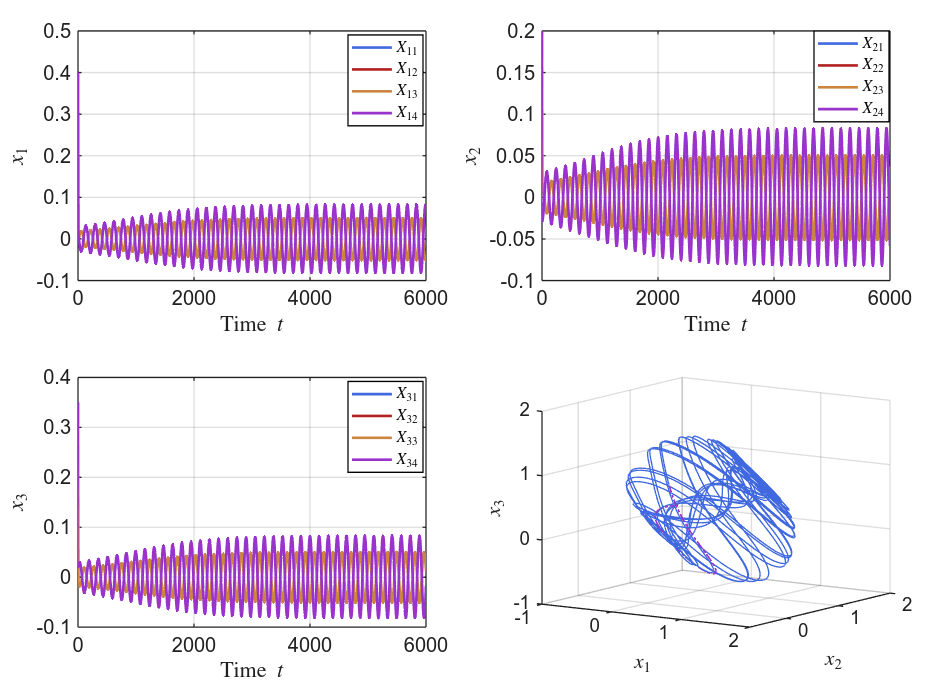} 
    \caption{Same setup as in Fig. \ref{fig:Unstable_Sprott__sub_critical}. Here however, the  coupling parameter is set to $\alpha = \alpha_c + \frac{1}{1000} = -0.6535$. The resulting network system is \emph{unstable} as predicted by Proposition \ref{prop:Critical_Stabilizing_coupling_val}. The trajectories are hinting that the dynamics of each  channel is evolving on a toroidal manifold, so it's either \emph{periodic} or \emph{quasi-periodic}.}
    \label{fig:Unstable_Sprott__super_critical}
 \end{figure}

Finally, we repeated the experiment above, but now changing the  coupling parameters to $\alpha = 1$ and $ \beta = -1$. Under these conditions, the hypotheses of Proposition \ref{prop:NonPosDivergenceGraphs} are not met. Consequently, stability of the switched system is \emph{not guaranteed} since the result only gives sufficient conditions. Nevertheless,  Fig. \ref{fig:cubedScalar_a=1_b=-1__unstable}. shows that the switched network system is unstable.

\addtolength{\textheight}{-.1cm}   

\begin{figure}[]
    \centering
    \includegraphics[width=0.80\linewidth]{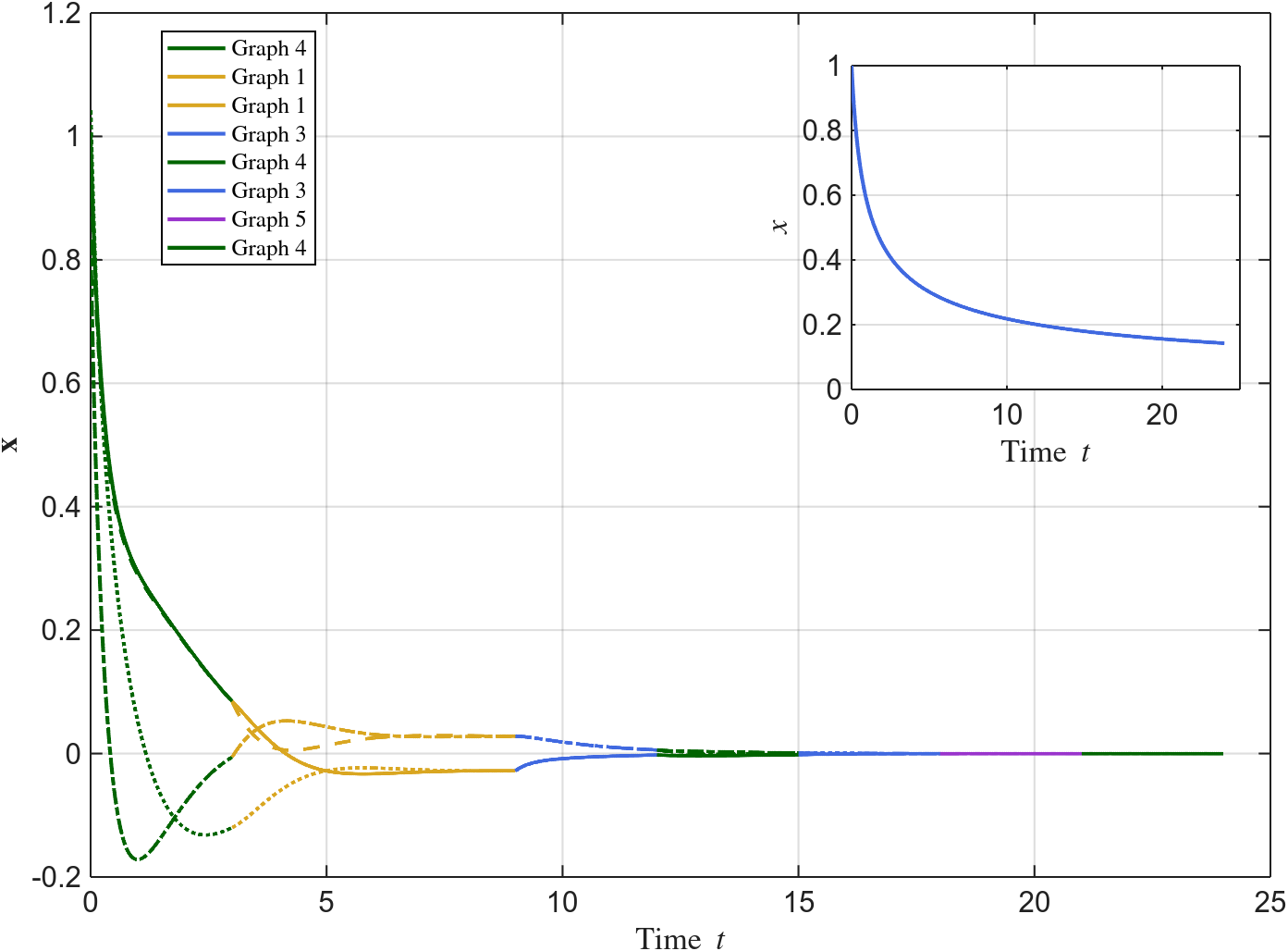} 
    \caption{Switched network system with modes made of scalar node dynamics $\dot{x} = -x^3$ (see inset plot) \emph{signless-diffusively}  coupled over the network topologies shown in Fig. \ref{fig:non-positive-Divergence-Graphs}. $\alpha = -1 = \beta$. Solid lines indicate the dynamics of Node 1, dashed lines that of Node 2, dotted lines indicate Node 3, and finally dash-dotted lines indicate Node 4. The hypothesis of Proposition \ref{prop:NonPosDivergenceGraphs} are met, and the switched network is indeed asymptotically  stable. } 
    \label{fig:cubedScalar_a=-1_b=-1__stable}
 \end{figure}

\begin{figure}[]
    \centering
    \includegraphics[width=0.5\linewidth]{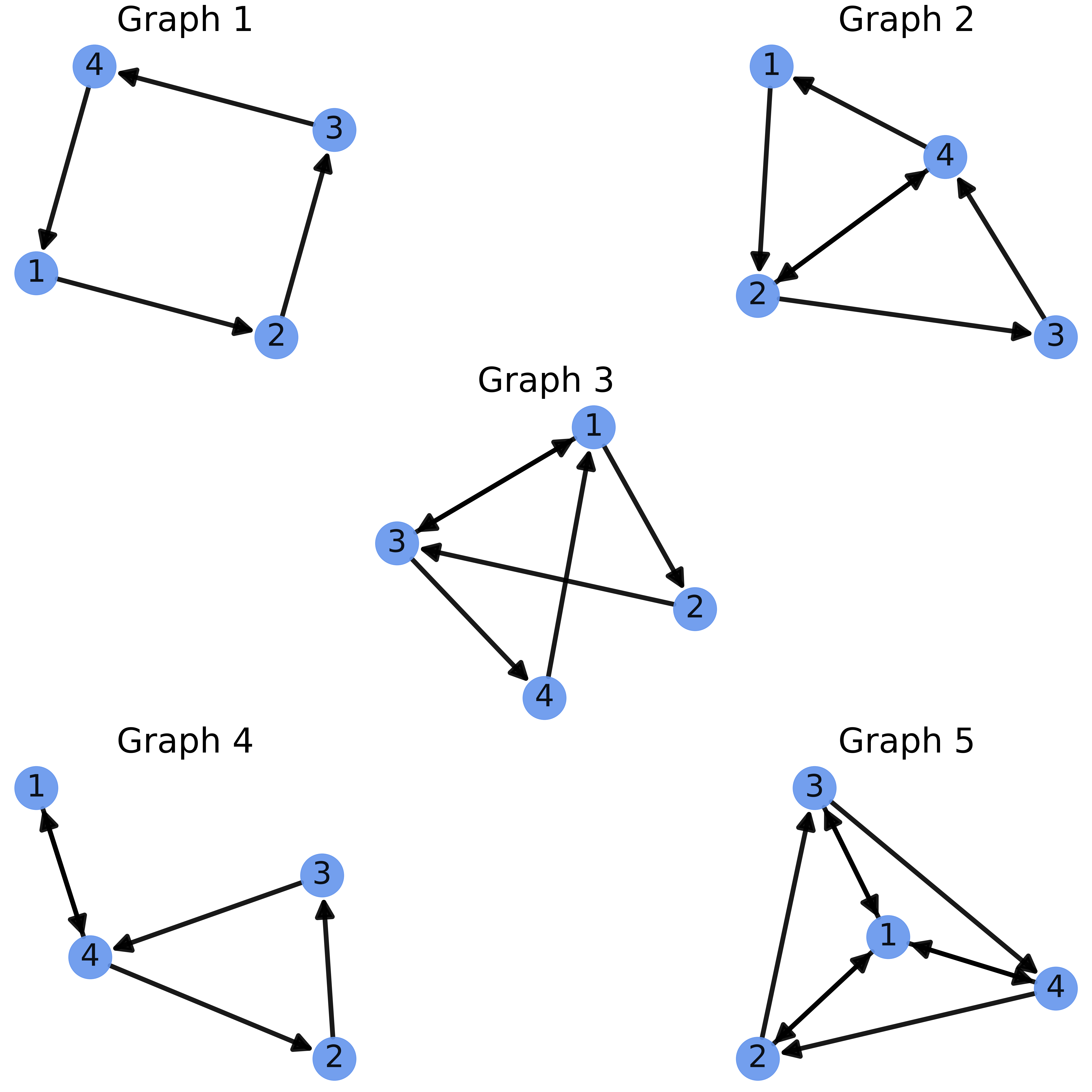} 
   \caption{Randomly selected set of $5$ \emph{non-positive divergence} graphs on $4$ nodes. These are directed graphs for which the out-degree of each node is no greater than the in-degree. This class of graphs is narrower than that of strongly connected graphs, but it's larger than that of undirected graphs. According to Proposition \ref{prop:NonPosDivergenceGraphs},  this is the most conducive class of network topologies for asymptotic stability under arbitrary switching of signless-diffusively coupled network systems.}
    \label{fig:non-positive-Divergence-Graphs}
 \end{figure}

 \begin{figure}[]
    \centering
    \includegraphics[width=0.80\linewidth]{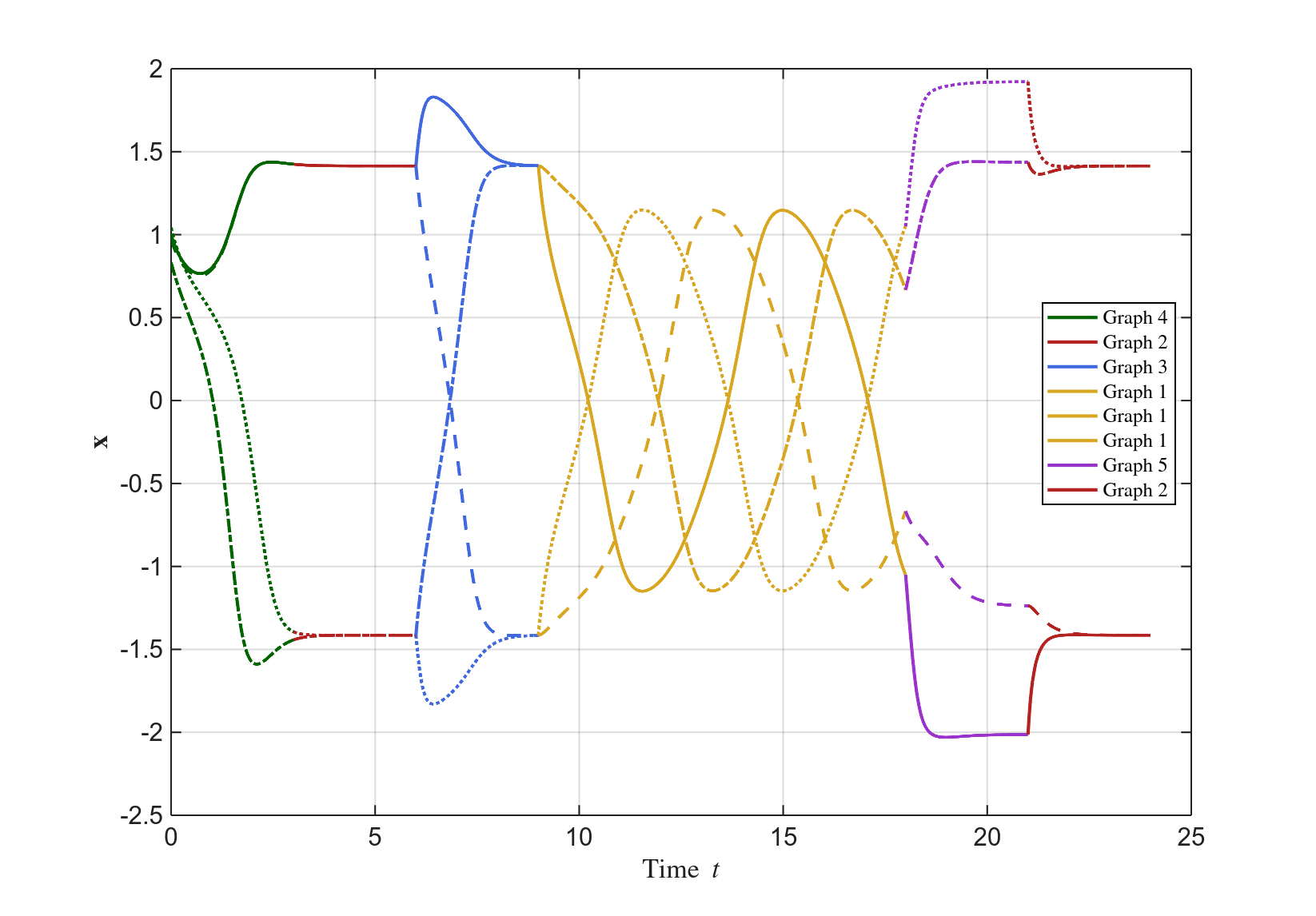} 
    \caption{Switched network system with modes made of scalar node dynamics $\dot{x} = -x^3$  \emph{diffusively} coupled over the network topologies shown in Fig. \ref{fig:non-positive-Divergence-Graphs}. $\alpha = 1$ and   $\beta= -1$. The line style convention is the same as in Fig. \ref{fig:cubedScalar_a=-1_b=-1__stable}. Here, the hypothesis of Proposition \ref{prop:NonPosDivergenceGraphs} are \emph{violated}, so the stability of the resulting switched system under arbitrary switching is \emph{not guaranteed} . Coincidentally, the particular switching signal used in this plot shows that the switched network system is in fact not stable.}
    \label{fig:cubedScalar_a=1_b=-1__unstable}
 \end{figure}

\section{Conclusion}

This paper highlights the critical role of coupling in the stability and stabilization of diffusively coupled network dynamical systems. Using Lyapunov methods, we derive expressions for key coupling parameters to control the network system stability. Furthermore, we derive sufficient conditions for ensuring  asymptotic stability of the network system under arbitrary switching. Our analysis demonstrates that appropriate coupling can synchronize nodes and enhance system stability, even in the presence of complex interactions and dynamic connectivity topologies. Numerical simulations support the findings, thus highlighting the importance of designing coupling strategies to ensure resilience in real-world systems.  A promising area of future research is the extension of our results to hypernetworks, namely  networks in which more than two agents can interact simultaneously.


\end{document}